\documentclass[11pt,a4paper,oneside]{amsart}

\usepackage{amsmath} 
\usepackage{amssymb} 
\usepackage{amsthm} 
\usepackage[english]{babel} 
\usepackage[font=small,justification=centering]{caption} 
\usepackage[nodayofweek]{datetime}
\usepackage[shortlabels]{enumitem} 
\usepackage[T1]{fontenc} 
\usepackage[a4paper]{geometry} 
\usepackage[utf8]{inputenc} 
\usepackage{ifthen} 
\usepackage{mathabx} 
\usepackage{mathtools} 
\usepackage[dvipsnames]{xcolor} 
\usepackage[pdftex,  colorlinks=true]{hyperref} 
    \hypersetup{urlcolor=RoyalBlue, linkcolor=RoyalBlue,  citecolor=black}
\usepackage{setspace} 
\usepackage{tikz-cd} 
\usepackage{xfrac} 

\usepackage{amscd}
\usepackage{cleveref}

\makeatletter
\def\@tocline#1#2#3#4#5#6#7{\relax
  \ifnum #1>\c@tocdepth 
  \else
    \par \addpenalty\@secpenalty\addvspace{#2}%
    \begingroup \hyphenpenalty\@M
    \@ifempty{#4}{%
      \@tempdima\csname r@tocindent\number#1\endcsname\relax
    }{%
      \@tempdima#4\relax
    }%
    \parindent\z@ \leftskip#3\relax \advance\leftskip\@tempdima\relax
    \rightskip\@pnumwidth plus4em \parfillskip-\@pnumwidth
    #5\leavevmode\hskip-\@tempdima
      \ifcase #1
       \or\or \hskip 1em \or \hskip 2em \else \hskip 3em \fi%
      #6\nobreak\relax
    \dotfill\hbox to\@pnumwidth{\@tocpagenum{#7}}\par
    \nobreak
    \endgroup
  \fi}
\makeatother

\newtheorem{thm}{Theorem}[section]
\newtheorem{theorem}[thm]{Theorem} \newtheorem{proposition}[thm]{Proposition} 
\newtheorem{lemma}[thm]{Lemma}
\newtheorem{corollary}[thm]{Corollary}
\newtheorem{prop}[thm]{Proposition}

\newtheorem{thmx}{Theorem}
\newtheorem{corx}[thmx]{Corollary}
\newtheorem{propx}[thmx]{Proposition}

\theoremstyle{definition}

\newtheorem{definition}[thm]{Definition}
\newtheorem{remark}[thm]{Remark}
\newtheorem{question}[thm]{Question}

\newtheorem{example}[thm]{Example}

\DeclareMathOperator{\Comm}{\mathrm{Comm}}

\DeclareMathOperator{\dist}{\mathsf{dist}}
\DeclareMathOperator{\Hdist}{\mathsf{hdist}}

\newcommand{\cala}{{\mathcal{A}}}

\newcommand{\calh}{{\mathcal{H}}}

\newcommand{\calp}{{\mathcal{P}}}
\newcommand{\calP}{{\mathcal P}}
\newcommand{\calq}{{\mathcal{Q}}}

\newcommand{\ZZ}{\mathbb{Z}}

\newcommand*{\Ad}{\textrm{Ad}}

\usepackage{tikz}
\usetikzlibrary{arrows,quotes}
\tikzstyle{blackNode}=[fill=black, draw=black, shape=circle]
\usetikzlibrary{decorations.pathmorphing}
\tikzset{snake it/.style={decorate, decoration=snake}}
\usetikzlibrary{calc}

\usepackage[style=alphabetic, citestyle=alphabetic, sorting=nyvt]{biblatex}
\AtEveryBibitem{\ifentrytype{article}{\clearfield{issn}}{}}
\AtEveryBibitem{\ifentrytype{article}{\clearfield{url}}{}}
\AtEveryBibitem{\ifentrytype{incollection}{\clearfield{url}}{}}
\NewBibliographyString{toappear}
\DefineBibliographyStrings{english}{%
  toappear = {to appear},
}

\renewbibmacro*{in:}{%
  \iffieldundef{pubstate}
    {}
    {\printfield{pubstate}%
     \setunit{\addspace}%
     \clearfield{pubstate}}%
  \printtext{%
    \bibstring{in}\intitlepunct}}
\usepackage{csquotes}

\title{Hyperbolically embedded subgroups and quasi-isometries of pairs}
\usepackage[foot]{amsaddr}
\author{Sam Hughes}
\email{sam.hughes@maths.ox.ac.uk}
\address[Sam Hughes]{Mathematical Institute, Andrew Wiles Building, University of Oxford, Oxford, OX2 6GG, UK}
\author{Eduardo Mart\'inez-Pedroza}
\address[Eduardo Mart\'inez-Pedroza]{Department of Mathematics and Statistics, Memorial University of Newfoundland, St. John's, NL, Canada}
\email{emartinezped@mun.ca}
\date{\today}


\begin{document}

\maketitle
\begin{abstract}
    We give technical conditions for a quasi-isometry of pairs to preserve a subgroup being hyperbolically embedded.  We consider applications to the quasi-isometry and commensurability invariance of acylindrical hyperbolicity of finitely generated groups.
    
\end{abstract}

\section{Introduction}
A group $G$ is \emph{acylindrically hyperbolic} if it admits a non-elementary, acylindrical action on a hyperbolic space.  An alternative characterisation is that $G$ is acylindrically hyperbolic if and only if $G$ contains a \emph{hyperbolically embedded subgroup} $H$, denoted $H\hookrightarrow_h G$, we will give a characterisation from \cite{MPR2021} in \Cref{prop:HypEmbHatGamma}.   

The class of acylindrically hyperbolic groups generalises the classes of non-elementary hyperbolic and relatively hyperbolic groups whilst sharing many similar properties \cite{Osin2016}.  In spite of this there are still foundational questions that remain open, for instance, it is known that a group being hyperbolic or relatively hyperbolic is invariant under quasi-isometry \cite{Gr1987} \cite{Dr2009}, but the corresponding question for acylindrical hyperbolicity is still open.

\begin{question} \cite[Question~2.20(a)]{Osin2019} \label{question.1}
Is the class of finitely generated acylindrically hyperbolic groups closed under quasi-isometry?
\end{question}

Some partial results are known, for instance acylindrical hyperbolicity passes to finite-index subgroups and is preserved by quotienting out a finite normal subgroup \cite{MiOs15}.  If the group is $\cala\calh$-accessible then acylindrical hyperbolicity can be passed to finite extensions \cite{MiOs15c}.  The property of being $\cala\calh$-accessible also passes to finite-index overgroups \cite{Ba2020}.  However, not every finitely presented acylindrically hyperbolic group is $\cala\calh$-accessible \cite[Theorem~2.18]{ABO2019}. Some experts in the field do not expect a complete positive answer to \Cref{question.1}.

This article relies on the notion of quasi-isometry of pairs, and our  results  provide technical conditions to ensure a quasi-isometry of pairs carries the property of being a hyperbolically embedded subgroup.

\begin{definition}[Quasi-isometry of  pairs]\label{defn:quasi-isometry-pairs}
Let $X$ and $Y$ be metric spaces, let $\mathcal{A}$ and $\mathcal{B}$ be   collections of subspaces of $X$ and $Y$ respectively.  A  quasi-isometry $q\colon X\to Y$ is a \emph{quasi-isometry  of pairs} $q\colon (X,\mathcal{A}) \to (Y,\mathcal{B})$ if there is $M>0$:
\begin{enumerate}
    \item For any $A\in \mathcal{A}$, 
    the set $\{ B\in \mathcal{B} \colon \Hdist_Y(q(A), B) <M \}$ is non-empty.      
    \item For any $B\in \mathcal{B}$, 
    the set $\{ A\in \mathcal{A} \colon \Hdist_Y(q(A), B) <M \}$ is non-empty. 
\end{enumerate}
In this case, if $q\colon X \to Y$ is a $(L,C)$-quasi-isometry, then $q\colon (X, \mathcal{A})\to (Y,\mathcal{B})$ is called a \emph{$(L,C,M)$-quasi-isometry}. If there is a quasi-isometry of pairs  $(X,\mathcal{A}) \to (Y,\mathcal{B})$ we say  that $(X,\mathcal{A})$ and  $(Y,\mathcal{B})$ are \emph{quasi-isometric pairs}.
\end{definition}

We specialise the previous definition to the case of   finitely generated groups with finite collections of subgroups as follows.

\begin{definition}[Quasi-isometry of group pairs]
Consider two pairs $(G, \mathcal{P})$ and $(H, \mathcal{Q})$ where $G$ and $H$ are finitely generated groups with chosen word metrics $\dist_G$ and $\dist_H$. Denote the Hausdorff distance between subsets of $H$ by $\Hdist_H$.
An $(L,C)$-quasi-isometry $q\colon G \to H$ is an \emph{$(L,C,M)$-quasi-isometry of  pairs}  $q\colon (G, \mathcal{P})\to (H, \mathcal{Q})$ if the relation
\begin{equation*}\label{eq:def-qi-relation} 
\dot{q} = \{ (A, B)\in G/\mathcal{P} \times H/\mathcal{Q} \colon \Hdist_H(q(A), B) <M  \} \end{equation*}
satisfies that the projections into $G/\mathcal{P}$ and $H/\mathcal{Q}$ are surjective. 
\end{definition} 

\begin{example}[Quasi-isometry of pairs and finite extensions]
Let $H$ be a finite index normal subgroup of finitely generated group $G$, and let $\calq$ be a finite collection of subgroups of $H$. Then the inclusion $(H,\calq) \hookrightarrow (G,\calq)$ is a quasi-isometry of pairs if the collection  $\{hQh^{-1} \colon h\in H \text{ and } Q\in \calq\}$ is invariant under conjugation by $G$, see \Cref{prop:FiniteIndexSupergroupIff}.
\end{example}

Recall that the \emph{commensurator} of a subgroup $P$ of a group $G$ is the subgroup \[Comm_G(P)=\{g\in G \colon P\cap gPg^{-1} \text{ is a finite index subgroup of $P$ and $gPg^{-1}$}\}.\]

\begin{definition}[Refinements]
 Let $\calp$ be a collection of subgroups of group $G$.  A \emph{refinement} $\calp^\ast$ of $\calp$ is a set of representatives of conjugacy classes of the collection of subgroups \[\{\Comm_G(gPg^{-1}) \colon P\in\calp \text{ and } g\in G \}.\]
 \end{definition} 

\begin{example}[Refinements and qi of pairs]
Let $\calq$ be a finite collection of subgroups of a finitely generated group $H$ and let $\calq^\ast$ be a refinement. If each $Q\in\calq$ is finite index in $\Comm_H(Q)$ then the identity map on $G$ is a quasi-isometry of pairs $(H,\calq)\to (H,\calq^\ast)$. 
\end{example}

\begin{example}[Refinements and finite extensions] \label{ex:refinement}
Let $A$ be a group, let  $\mathcal{H}$ be an almost malnormal collection of infinite subgroups, and let $F\leq \mathsf{Aut}(A)$ be a finite subgroup.  If $F$ acts freely on $\mathcal{H}$ and $\mathcal{H}_F$ is a collection of representatives of $F$-orbits in $\mathcal{H}$, then a refinement of $\mathcal{H}$ in $A\rtimes F$ is $\mathcal{H}_F$.
\end{example}

\begin{definition}[Reduced collections]
 A collection of subgroups $\calp$ of a group $G$ is \emph{reduced} if  
for any $P,Q\in \calp$ and $g\in G$, if $P$ and $gQg^{-1}$ are commensurable then  $P=Q$ and $g\in P$.
\end{definition}


Our first result, \Cref{propx:summary},  describes a strategy to obtain positive results to \Cref{question.1}.  For a group $G$ with a generating set $S$, let $\Gamma(G,S)$ denote the corresponding \emph{Cayley graph}, see \Cref{def:CayleyGraph}. 
 
\begin{thmx}[\Cref{prop:summary}]\label{propx:summary} 
Let $q\colon G\to H$ be a quasi-isometry of finitely generated groups, let $\calp$ and $\calq$ be finite collections of subgroups of $G$ and $H$ respectively, and let $S$ and $T$ be (not necessarily finite) generating sets of $G$ and $H$ respectively. Suppose 
\begin{enumerate}
    \item $q\colon (G,\calp) \to (H,\calq)$ is a quasi-isometry of pairs, and
    \item $q\colon  \Gamma(G,S) \to  \Gamma(H,T)$ is a quasi-isometry.
\end{enumerate}
The following statements hold:
\begin{enumerate}
    \item If $\calp$ and $\calq$ are reduced collections in $G$ and $H$ respectively; then $\calp\hookrightarrow_h (G,S)$ if and only if $\calq\hookrightarrow_h (G,T)$.
    \item If $\calq$ contains only infinite subgroups and  $\calq\hookrightarrow_h (H,T)$ then $\calp^*\hookrightarrow_h (G,S)$.
\end{enumerate}
\end{thmx}

\subsection*{Qi-characteristic collections}
The first numbered hypothesis of \Cref{propx:summary} raises the following general problem:  Given a finite collection of subgroups $\calq$ of a group $H$ and a quasi-isometry $q\colon G\to H$ of finitely generated groups, is there a collection $\calp$ of subgroups of $G$ such that $q\colon(G,\calp)\to(H,\calq)$ is a quasi-isometry of pairs?

This problem was studied in~\cite{MaSa21} where the notion of qi-characteristic collection is introduced and it is proved that if the collection $\calq$ is qi-characteristic in $H$, then any quasi-isometry of finitely generated groups induces a collection $\calp$. 

\begin{definition}[Qi-characteristic]\cite{MaSa21}\label{thm:qi-characteristic-groups2}
A collection of subgroups $\mathcal{P}$ of a finitely generated group $G$ is  \emph{quasi-isometrically characteristic}  (or shorter  \emph{qi-characteristic})  if $\mathcal{P}$ is finite; each $P\in \mathcal{P}$ has finite index in its commensurator; and for every $L\geq1$ and $C\geq0$ there is $M=M(G, \mathcal{P}, L,C)\geq0$ such that every  $(L,C)$-quasi-isometry $q\colon G\to G$ is an $(L,C,M)$-quasi-isometry of pairs $q\colon (G,\mathcal{P})\to (G,\mathcal{P})$.  \end{definition}

\begin{example}\label{ex qi-char}
The argument by  Behrstock, Dru\c{t}u and Mosher proving quasi-isometric rigidity of relative hyperbolicity with respect to non-relatively hyperbolic groups (NRH groups) shows that if $H$ is hyperbolic group relative to a collection  $\calq$ of NRH subgroups, then  $\calq$ is qi-characteristic~\cite[Theorems 4.1 and 4.8]{BDM09}. Another example is provided by mapping class groups. Ruling out a few surfaces of low complexity,  any self quasi-isometry of the mapping class group  is at uniform distance from left multiplication by an element of the group, see the work of Behrstock, Kleiner, Minsky and Mosher~\cite[Theorem 1.1]{BKMM12}. As a consequence, the hyperbolically embedded (virtually cyclic) subgroup generated by a pseudo-Anosov is qi-characteristic. More generally, any finite collection of subgroups of such mapping class groups are qi-characteristic. 
\end{example}

\begin{corx} \label{thmx.speculation}
 Let $G$ and $H$ be finitely generated groups, let $T$ be a generating set of $H$, let
  $\calq$ be a finite collection of  subgroups of $H$ such that $\calq\hookrightarrow_h (H,T)$, and let $q\colon G \to H$ be a quasi-isometry.
  If    
  \begin{enumerate}
      \item $\calq$ is a qi-characteristic collection of subgroups of $H$, and
      \item there is a generating set $S\subset G$ such that $q\colon \Gamma(G,S) \to \Gamma(H,T)$ is a quasi-isometry;
  \end{enumerate}
  then there is a finite collection $\calp$ of subgroups of $G$ such that $\calp\hookrightarrow_h (G,S)$ and $q\colon (G,\calp) \to (H,\calq)$ is a quasi-isometry of pairs.
\end{corx}
\begin{proof}
Without loss of generality, assume that all subgroups in $\calq$ are proper infinite subgroups. Note that removing finite subgroups from $\calq$ preserves being qi-characteristic and  that $\calq\hookrightarrow_h (H,T)$. On the other hand, if $\calq$ contains $H$, then the theorem is trivial by taking $\calp$ the collection that contains only $G$ and $S$ any finite generating set of $G$.  Since $\calq$ is qi-characteristic, the quasi-isometry $q\colon G\to H$ induces a finite collection $\calp$ such that $q\colon (G,\calp)\to (H,\calq)$ is a quasi-isometry of pairs, this is precisely~\cite[Theorem 1.1]{MaSa21}.  Then the second statement of \Cref{propx:summary} and $\calq\hookrightarrow_h (H,T)$ imply that  $\calp^\ast\hookrightarrow_h (G,S)$.
\end{proof}

\subsection*{Uniform Quasi-actions}

The second numbered hypothesis of \Cref{propx:summary} raises the problem: Given a group $H$ with a generating set $T$ and a quasi-isometry $q\colon G\to H$ of finitely generated groups, is there a generating set $S\subset G$ such that $q\colon \Gamma(G,S) \to \Gamma(H,T)$ is a quasi-isometry of Cayley graphs?

We show that a positive answer to this question is equivalent to asking that the quasi-action of $G$ on $H$ induced by $q$ is $T$-uniform in the following sense, see \Cref{lemx:0423-01}. 

\begin{definition}[Uniform induced quasi-action]\label{def:UniformT}
Let $G$ and $H$ be finitely generated groups and let $q\colon G\to H$ be a quasi-isometry with quasi-inverse $\bar q$. Let $T\subset H$ be a generating set (possibly infinite). We say that the quasi-action of $G$ on $H$ induced by $q$ is \emph{uniform with respect to $T$} if there are constants $L\geq 1$, $C\geq0$ such that for each $g\in G$ the function $q_g\colon H\to H$ given by $q_g(h)=  q(g\cdot\bar q(h))$ is an $(L,C)$-quasi-isometry
$q_g\colon \Gamma(H,T)\to \Gamma(H,T)$.
\end{definition}

\begin{example}[Uniform quasi-action and finite extensions]
Let $H$ be a finite index normal subgroup of finitely generated group $G$ and let $T$ be a generating set of $H$ invariant under conjugation by $G$. The $G$-action by conjugation on $H$ preserves the word metric induced by $T$. On the other hand, any transversal $R$ of $H$ in $G$ induces a quasi-isometry  $q\colon G \to H$ given by $q(hg)=h$ for $h\in H$ and $g\in R$. In this case the quasi-action of $G$ on $H$ induced by $q$ is uniform with respect to $T$, see  \Cref{lem:Friday}.
\end{example}

\begin{propx}[\Cref{lem:0423-01}] \label{lemx:0423-01}
Let $G$ and $H$ be groups with finite generating sets $S_0$ and $T_0$, and let $q\colon \Gamma(G,S_0) \to \Gamma(H,T_0)$ be a quasi-isometry. 
Let $T\subset H$ containing $T_0$. The following statements are equivalent:
\begin{enumerate}
    \item The  quasi-action of $G$ on $H$ induced by $q$ is uniform with respect to $T$.
    \item There is $S\subset G$ containing $S_0$ such that $q\colon \Gamma(G,S) \to \Gamma(H,T)$ is a quasi-isometry. 
\end{enumerate}
\end{propx}

\begin{corx}  \label{thmx:HypEmbd}
Let $G$ and $H$ be finitely generated groups with finite collections of infinite subgroups $\calp$ and $\calq$ respectively. Suppose $q\colon (G,\calp) \to (H,\calq)$ is a quasi-isometry of pairs inducing a $T$-uniform  quasi-action of $G$ on $H$.
If $\calq\hookrightarrow_h(H,T)$, then $\calp^\ast \hookrightarrow_h G$.
\end{corx}  
\begin{proof}
Since the quasi-action of $G$ on $H$ induced by $q$ is $T$-uniform, \Cref{lemx:0423-01} implies that there is a generating set $S$ of $G$ such that $q\colon \Gamma(G, S) \to \Gamma(H,T)$ is a quasi-isometry.
Then the second statement of  \Cref{propx:summary} and $\calq\hookrightarrow_h (H,T)$  imply that $\calp^\ast\hookrightarrow_h (G,S)$.
 \end{proof}


Let us remark that for this last  corollary,  in the case that $T$ is finite, then there is a finite $S\subset G$ such that $\mathcal{P} \hookrightarrow_h (G, S)$; this case is implied by the results on quasi-isometric rigidity of relative hyperbolicity in~\cite{BDM09}. 

\subsection*{Finite Extensions}

The following application is a particular instance of \Cref{cor:Saturday} in the main body of the article. 

\begin{thmx}[\Cref{cor:Saturday}]\label{corx:Saturday}
Let $H$ be a finite index normal subgroup of a finitely generated group $G$, and let $\calq$ be a finite collection of infinite subgroups of $H$ such that $\calq\hookrightarrow_h (H,T)$.  
Suppose:  
\begin{enumerate}
    \item  The set $T$ is invariant under conjugation by $G$.
    \item The collection  $\{hQh^{-1} \colon h\in H \text{ and } Q\in \calq\}$ is invariant under conjugation by $G$. 
\end{enumerate}
If $\calq^*$ is a refinement of $\calq$ in $G$, then $\calq^*\hookrightarrow_h G$.
\end{thmx}

\begin{example}
  Let
$G= \langle a,b,t \colon  tat^{-1}=b,\ t^2=1  \rangle\cong F_2\rtimes\ZZ_2$, let $H=\langle a,b\rangle$, and let $\calq=\{ \langle a\rangle, \langle b\rangle\}$. Note that $\calq\hookrightarrow_h H$, and, for instance one can take $\calq^\ast = \{\langle a \rangle\}$ and observe that $\calq^* \hookrightarrow_h G$. In contrast,  for  $\calq_0=\{ \langle a \rangle\} \hookrightarrow_h H$ the theorem does not apply since the conjugates of $\langle a\rangle$ in $H$ are not invariant under conjugation by elements of $G$.
\end{example}

The next result illustrates concrete examples were \Cref{corx:Saturday} applies.

\begin{thmx}[\Cref{thm.hypemb.semidirect}]\label{thmx.hypemb.semidirect}
Let $A$ be a {finitely generated} group with a (not necessarily finite) generating set $T$, and let $\mathcal{H}$ be a finite collection of  infinite subgroups such that $\mathcal{H}\hookrightarrow_h (A,T)$.
If $F\leq \mathsf{Aut}(A)$ is   finite,  $T$ and $\mathcal{H}$ are $F$-invariant, and the $F$-action on $\mathcal{H}$ is free,  then $\calh_F\hookrightarrow_h(A\rtimes F,T\cup F)$ where $\calh_F$ is collection of representatives of $F$-orbits in $\calh$.
\end{thmx}

\begin{example}
Let $A=\bigast_{i=1}^n B_i$ with each $B_i$ isomorphic to a fixed finitely generated group $B$.  Let $F=\ZZ_n$ act on $A$ by cyclically permuting the copies of $B$.  Consider the generating set of $A$ given by $T=\bigcup_{i=1}^nB_i\backslash\{1\}$, then $T$ is $F$-invariant.  Now, the collection $\calh=\{B_1,\dots,B_n\}$ is hyperbolically embedded into $(A,T)$ and $F$ acts freely by conjugation on $\calh$.  All of the hypotheses of the previous theorem have been verified so we conclude that $B_1\hookrightarrow_h(A\rtimes F,T\cup F)$.
\end{example}



\subsection*{Organization.}   The rest of the article is divided into five sections.   \Cref{sec:quasiActions} is on quasi-actions, it contains the proof of \Cref{lemx:0423-01} as well as some corollaries. The proof of \Cref{propx:summary} is the content of \Cref{sec:Summary}. Then Sections~\ref{sec:ThmE} and~\ref{sec:ThmF} contain the proofs of \Cref{corx:Saturday} and \Cref{thmx.hypemb.semidirect} respectively. Finally, \Cref{sec conclusion} contains some questions and discussion about related to the results in this article and the definition of a quasi-isometry of pairs.

\subsection*{Acknowledgements}
The first author would like to thank his PhD supervisor Professor Ian Leary.  The first author was supported by the Engineering and Physical Sciences Research Council grant number 2127970.  This work has received funding from the European Research Council (ERC) under the European Union’s Horizon 2020 research and innovation programme (Grant agreement No. 850930). The second author acknowledges funding by the Natural Sciences and Engineering Research Council of Canada, NSERC.  Both authors would like to thank Luis Jorge S\'anchez Salda\~na and the anonymous referee for their helpful comments.

\section{Uniform quasi-actions}

\label{sec:quasiActions}

\begin{definition}[Uniform quasi-action]
Let $G$ be a group and let $X$ be a metric space. Let $\mathsf{QI(X)}$ denote the set of quasi-isometries $X\to X$. 
A function $G\to \mathsf{QI(X)}$, $g\mapsto f_g$, is a \emph{quasi-action} if there is $K\geq 0$  such that for any $g_1,g_2\in G$  
\begin{enumerate}
    \item the map $f_{g_1g_2}$ is at distance at most $K$ from the map $f_{g_1}\circ f_{g_2}$ in the $L_\infty$-distance, and
    \item the map $f_{g_1}\circ f_{g_1^{-1}}$ is at distance at most $K$ from the identity.
\end{enumerate}
The quasi-action $G\to \mathsf{QI}(X)$ is \emph{uniform} if there are constants $L\geq 1$ and $C\geq0$ such that for any $g\in G$ the map $f_g$ is an $(L,C)$-quasi-isometry. 
\end{definition}

It is well known that a quasi-isometry $q\colon G\to H$ of finitely generated groups induces a uniform quasi-action of $G$ on $H$:

\begin{definition}[Uniform quasi-action induced by a quasi-isometry]\label{lem:quasiAction}
Let $G$ be a group with a word metric  induced by a finite generating set, let $X$ be a metric space,
let $q\colon G \to X $ and $\bar q\colon X \to G$ be $(L_0,C_0)$-quasi-isometries such that $q\circ \bar q$ and $\bar q\circ  q$ are at distance less than $C_0$ from the identity maps on $X$ and $G$ respectively. 
For $g\in G$, let 
\[ L_g\colon G\to G, \qquad  x\mapsto gx;\] 
and let 
\[ q_g\colon X\to X\qquad q_g= q \circ \mathbf{g}\circ\bar q.\] 
It is an exercise to verify that there are constants $L\geq 1$ and $C\geq0$ such that: 
\begin{itemize}
    \item For $g\in G$, $q_g\colon X\to X$ is an $(L,C)$-quasi-isometry. 
    \item ($G$ quasi-acts on $X$)  For $g_1,g_2\in G$, the map $q_{g_1g_2}$ is at distance at most $C$ from the map $q_{g_1}\circ q_{g_2}$; and the map $q_{g_1}\circ q_{g_1^{-1}}$ is at distance at most $C$ from the identity.
    \item ($G$ acts $C_0$-transitively on $X$) For every $x ,y\in X$ there is $g\in G$ such that $\dist_G(x, q_g(y))\leq  C$.
\end{itemize}
The map $G\to \mathsf{QI}(X)$ given by $g\mapsto q_g$ is called the \emph{uniform quasi-action of $G$ on $X$ induced by $q$ and $\bar q$}. 
\end{definition}

\begin{remark}[Equivalence of  Definitions~\ref{lem:quasiAction} and~\ref{def:UniformT} ]\label{rem:Uniform}
In the context of \Cref{def:UniformT}, if the induced quasi-action of $G$ on $H$ is uniform with respect to $T$, then $G\to \mathsf{QI}(\Gamma(H,T))$ given by $g\mapsto q_q$ is a uniform quasi-action in the sense of \Cref{lem:quasiAction}. Indeed, since $T$ contains a finite generating set of $H$ there is $M>0$ such that $\dist_{(H,T)}\leq  M\dist_{(H,T_0)}$. Hence if two functions $H\to H$ are at finite $L_\infty$-distance with respect to $\dist_{(H,T_0)}$, then the same holds for $\dist_{(H,T)}$. 
\end{remark}

\begin{definition}[Cayley Graph]\label{def:CayleyGraph}
Let $G$ be a group with a  generating set $S$. The \emph{Cayley graph $\Gamma(G,S)$ of $G$ with respect to $S$} is the $G$-graph with vertex set $G$ and edge set $\{ \{g, gs\}\colon g\in G,\ s\in S  \}$. 
\end{definition}

\begin{proposition}[\Cref{lemx:0423-01}] \label{lem:0423-01}
Let $G$ and $H$ be groups with finite generating sets $S_0$ and $T_0$, and let $q\colon \Gamma(G,S_0) \to \Gamma(H,T_0)$ be a quasi-isometry. 
Let $T\subset H$ containing $T_0$. The following statements are equivalent:
\begin{enumerate}
    \item The  quasi-action of $G$ on $H$ induced by $q$ is uniform with respect to $T$.
    \item There is $S\subset G$ containing $S_0$ such that $q\colon \Gamma(G,S) \to \Gamma(H,T)$ is a quasi-isometry. 
\end{enumerate}
\end{proposition}
\begin{proof} 
That the second statement implies the first one is immediate. Conversely, suppose that $q$ and $\bar q$ are $(L_0,C_0)$-quasi-isometries $\Gamma(G,S_0)\to \Gamma(H,T_0)$ and $\Gamma(H,T_0)\to \Gamma(G,S_0)$ respectively. Without loss of generality assume that $q(e)=e$ and $\bar q(e)=e$ where $e$ denotes the identity in each corresponding group. 

Let $K_0=L_0+C_0+1$ and define 
\[
\begin{split}
S=  \{ f^{-1}g \in G \colon & \text{ there are }  h\in H   \text{ and }  t\in T \text{ such that } \\  & \dist_{(H, T_0)}(q(f),h)\leq K_0  \text{ and }  \dist_{(H,T_0)}(q(g),ht)\leq K_0  \} .
\end{split}\]
Note that $S_0\subset S$ since $q\colon \Gamma(G,S_0)\to \Gamma(H,T_0)$ is an $(L_0,C_0)$-quasi-isometry. In particular, $S$ is a generating set of $G$. 

Let $L_1\geq 1$ and $C_1\geq 0$ be such that the $G$-action on $H$ induced by $q$ is $(L_1,C_1)$-uniform with respect to $T$. In particular, for every $g\in G$ the function $q_g\colon H\to H$ is an $(L_1,C_1)$-quasi-isometry $\Gamma(H,T) \to \Gamma(H,T)$.

Now, we prove that if the induced quasi-action of $G$ on $H$ is uniform with respect to $T$, then $q\colon \Gamma(G,S) \to \Gamma(H,T)$ is a quasi-isometry. Observe that every vertex of $\Gamma(H,T)$ is at distance at most $C_0$ from $q(G)$ with respect to $\dist_{(H,T_0)}$ and hence with respect to $\dist_{(H,T)}$. Below we prove inequalities~\eqref{eq:0422-01} and~\eqref{eq:0422-02} which will conclude proof.  

\textbf{Claim:} \emph{There is constant $\bar L$ such that 
\begin{equation}\label{eq:0422-01}\begin{split}
\dist_{(H,T)}(q(a),q(b)) 
\leq \bar L\dist_{(G,S)}(a,b).
\end{split}
\end{equation}
for any $a,b\in G$.}

\textbf{Proof of claim:} 
Let $s\in S$. Then there are $f,g\in G$, $h\in H$ and $t\in T$ such that $s=f^{-1}g$ and
\[ \dist_{(H,T_0)}(q(f),h)\leq K_0,\quad \dist_{(H,T_0)}(q(g),ht)\leq K_0.\]
It follows that
\[\dist_{(H,T)}(q_f(e),q_g(e))= \dist_{(H,T)}(q(f),q(g))\leq 2K_0+1.\]
Since the quasi-action of $G$ on $\Gamma(H,T)$ is $(L_1,C_1)$-uniform, the previous inequality implies that
\[
\begin{split}
\dist_{(H,T)}(e,q(s))  &= \dist_{(H,T)}(q_e(e),q_{f^{-1}g}(e)) \\    
& \leq L_1\dist_{(H,T)}(q_f\circ q_e(e),q_f\circ q_{f^{-1}g}(e))+ C_1\\
& \leq L_1\dist_{(H,T)}(q_f(e), q_{g}(e))+  3C_1\\
& \leq  L_1(2K_0+1)+3C_1=:\bar L_0.
\end{split}
\]
For any $g\in G$ and $s\in S$, we have that
\begin{equation*}
\begin{split}
  \dist_{(H,T)}(q(g),q(gs)) & = \dist_{(H,T)}(q_{g}(e), q_{gs}(e)) \\
  &\leq L_1 \dist_{(H,T)}(q_{g^{-1}}\circ q_{g}(e), q_{g^{-1}}\circ q_{gs}(e)) + C_1 \\
  &\leq L_1\dist_{(H,T)}(e, q_{g^{-1}gs}(e)) +3C_1\\
  & \leq L_1\dist_{(H,T)}(q(e),q(s)) + 3C_1.
\end{split}
\end{equation*}
and hence
\[ \dist_{(H,T)}(q(g),q(gs)) \leq  \dist_{(H,T_0)}(q(g),q(gs)) \leq \bar L \]
where $\bar L= L_1(\bar L_0)+3C_1$.  If $a,b\in G$ and $[u_0,\ldots, u_\ell]$ is a geodesic in $\Gamma(G,S)$ from $a$ to $b$, then the triangle inequality implies inequality~\eqref{eq:0422-01}. $\blackdiamond$

\textbf{Claim:}  \emph{For any $a,b\in G$ we have}
\begin{equation}\label{eq:0422-02}
\dist_{(G,S)}(a,b)\leq \dist_{(H,T)}(q(a),q(b)).    
\end{equation}

\textbf{Proof of claim:} Suppose that $[h_0,\ldots ,h_\ell]$ is a geodesic in $\Gamma(H,T)$ from $q(a)$ to $q(b)$. Since $q\colon \Gamma(G,S_0)\to \Gamma (H,T_0)$ is a $(L_0,C_0)$-quasi-isometry, for each $i$, there is $g_i\in G$ such that $\dist_{(H,T_0)}(q(g_i),h_i)\leq C_0$. Let $g_0=a$ and $g_\ell=b$. Observe that $g_i^{-1}g_{i+1}\in S$ for $0\leq i<\ell$, and hence $\dist_{(G,S)}(g_i,g_{i+1})\leq 1$. Now, $[g_0, \ldots , g_\ell]$ is a path in $\Gamma(G,S)$ from $a$ to $b$ and therefore $\dist_{(G,S)}(a,b)\leq \dist_{(H,T)}(q(a),q(b))$ proving  inequality~\eqref{eq:0422-02}. $\blackdiamond$
\end{proof}

\begin{corollary}\label{cor:SaturdayMorning01}
Let $G$ and $H$ be groups with finite generating sets $S_0$ and $T_0$.
Let $q\colon G\to H$ be a group homomorphism which is also an $(L_0,C_0)$-quasi-isometry $q\colon \Gamma(G,S_0) \to \Gamma(H,T_0)$.
If  $T\subset H$ contains $T_0$, then  there is $S\subset G$ containing $S_0$ such that $q\colon \Gamma(G,S) \to \Gamma(H,T)$ is a quasi-isometry.
\end{corollary}
\begin{proof}
Let $\bar q\colon H\to G$ be a quasi-inverse of $q$ and, by increasing $L_0$ and $C_0$ if necessary, assume that $\bar q\colon \Gamma(H,T_0) \to \Gamma(G, S_0)$ is a $(L_0,C_0)$-quasi-isometry.  Moreover, suppose $q\circ \bar q$ and $\bar q\circ q$ are at distance at most $C_0$ from the corresponding identity maps with respect to $\dist_{(H,T_0)}$ and $\dist_{(G,S_0)}$. Note that for any $g\in G$,
\[q_g (h) = q(g \cdot \bar q (h))=q(g) \cdot q (\bar q (h)). \]
Hence $q_g$ is an $(1,C_0)$-quasi-isometry since it is the composition of $q \circ \bar q$ followed by the isometry given by multiplication on the left by $q(g)$.  Then the proof concludes by invoking  \Cref{lem:0423-01}.
\end{proof}

The following result is  the particular case of  \Cref{cor:SaturdayMorning01} in which $H$ is a finite index subgroup of $G$. In this case, one can give a more algebraic description of the generating set $S$.  The proof follows the same lines as the previous argument modulo \Cref{lem:Friday}. 

\begin{proposition}\label{prop:FridayNight}
Let $H$ be a finite index normal subgroup of a finitely generated group $G$. Let $T$ be a generating set of $H$, let $R$ be a right transversal of $H$ in $G$, and let $S=T\cup R$. If the $G$-action by conjugation on $H$ is a uniform quasi-action on  $\Gamma(H,T)$, then the inclusion  $\Gamma(H,T) \hookrightarrow  \Gamma(G,S)$ is a quasi-isometry.
\end{proposition}

We divert the proof of the proposition  after the following lemma.

\begin{lemma}\label{lem:Friday}
Let $H$ be a finite index normal subgroup of a finitely generated group $G$. Let $T$ be a generating set of $H$ containing a finite generating set $T_0$, let $R$ be transversal of $H$ in $G$, let $S_0$ be a finite generating set of $G$, and let $q\colon \Gamma(G,S_0)\to \Gamma(H,T_0)$ be the quasi-isometry defined by $q(hg)=h$ for $h\in H$ and $g\in R$. The following statements are equivalent:
\begin{enumerate}
    \item The $G$-action by conjugation on $H$ is a uniform quasi-action on  $\Gamma(H,T)$.
    \item  The quasi-action of $G$ on $H$ induced by $q$ is uniform with respect to $T$.
\end{enumerate}
\end{lemma}
\begin{proof}
Take as the quasi-inverse of $q$ the inclusion $H\hookrightarrow G$. For $h\in H$, let $L_h\colon H\to H$ be given by $L_h(x)=hx$, i.e. multiplication on the left. Note that $L_h\colon \Gamma(H,T)\to \Gamma(H,T)$ is an isometry for every $h\in H$.  
 
Let $g\in G$ and suppose that $g=h_*g_*$ where $h_*\in H$ and $g_*\in R$. Then
\[q_g(h)=q(gh)=q(ghg^{-1}h_*g_*)=ghg^{-1}h_*=h_*g_*hg_*^{-1}h_*^{-1}h_*=h_*g_*hg_*^{-1}\]
and hence 
\[q_g = L_{h_*}\circ \Ad({g_*}),\]
where $\Ad(g_*)$ is conjugation by $g_*$.  It follows $q_g\colon \Gamma(H,T)\to \Gamma(H,T)$ is an $(L,C)$-quasi-isometry for all $g\in G$ if and only if $\Ad({g_*}) \colon \Gamma(H,T)\to \Gamma(H,T)$ is an $(L,C)$-quasi-isometry for all $g_*\in R$.  In particular, the first statement implies the second by \Cref{rem:Uniform}, and the second statement implies the first since the constants $L$ and $C$ hold for all conjugations.
\end{proof}

\begin{proof}[Proof of \Cref{prop:FridayNight}.]
Let $T_0 \subset T$ be a finite generating set of $H$, let $S_0=T_0\cup R$. Note that $S_0$ is a finite generating set of $G$. Then $q\colon \Gamma(G,S_0)\to \Gamma(H,T_0)$ is a $(L_0,C_0)$ quasi-isometry for some $L_0\geq 1$ and $C_0\geq0$, and  the quasi-inverse $\bar q$ can be taken as the inclusion $\Gamma(H,T_0)\hookrightarrow \Gamma(G,S_0)$.
 
Observe that in $\Gamma(G,S)$ the vertices $g=hr$ and $q(g)=h$ are adjacent since $r\in S$. Therefore, if $[v_0,\ldots ,v_\ell]$ is a geodesic path in $\Gamma(H,T)$ from $q(a)$ to $q(b)$, then $[a, v_0,\ldots ,v_\ell,b]$ is a path in $\Gamma(G,S)$ from $a$ to $b$, and  hence
\[ \dist_{(G,S)}(a,b) \leq  \dist_{(H,T)}(q(a),q(b))+2. \]

We now prove the other inequality.  Since the $G$-action on $H$ by conjugation is a uniform quasi-action on $\Gamma(H,T)$, \Cref{lem:Friday} implies that the quasi-action of $G$ on $H$ induced by $q$ is $(L_1,C_1)$-uniform with respect to $T$, for some $L_1\geq 1$ and $C_1\geq0$.

 Let $K_0=L_0+C_0+1$.  Observe that
\[
\begin{split}
S \subseteq \{ f^{-1}g \in G \colon & \text{ there are }  h\in H   \text{ and }  t\in T \text{ such that } \\  & \dist_{(H, T_0)}(q(f),h)\leq K_0  \text{ and }  \dist_{(H,T_0)}(q(g),ht)\leq K_0  \} .
\end{split}\]
Indeed, let $s\in S=T\cup R$, there are two cases.  First, if $s\in T$ let $f=h=e$ and $g=t=s$; and second if $s\in R$ let $f=h=e$, $g=s$ and $t$ any element of $T_0$. Then, exactly as in the first claim in the proof of \Cref{lem:0423-01}, one defines a constant $\bar L=\bar L(L_1,C_1,K_0)$ and deduces the inequality
\begin{equation}
\dist_{(H,T)}(q(a),q(b)) \leq \bar L\dist_{(G,S)}(a,b).
 \end{equation}

It remains to show
 \begin{equation} 
\dist_{(G,S)}(a,b)\leq \dist_{(H,T)}(q(a),q(b))+2.    
\end{equation} 
for any $a,b\in G$, concluding the proof. This is clear since $\Gamma(H,T)$ is a subgraph of $\Gamma(G,T)$ and $\dist_{G,S}(g,q(g))\leq 1$ for any $g\in G$.
\end{proof}

The following example by Minasyan and Osin illustrates the need for the hypothesis relating to the conjugation action in Corollary~\ref{prop:FridayNight}.

\begin{example}\cite{MiOs15c}\label{ex.MinasyanOsin}
Let $H=\langle a,b \rangle$ be the free group of rank two, let $G=\langle a,b,t \colon tat^{-1}=b,\quad t^2=e \rangle$, let $T=\{b,a,a^{-1},a^2,a^{-2},\ldots \}$ and $S=T\cup \{t\}$. The inclusion $\Gamma(H,T) \to \Gamma(G,S)$ is not a quasi-isometry. Indeed, in $G$ we have $ta^nt^{-1}=b^n$ and hence $\dist_{(G,S)}(e,b^n)=3$ but $\dist_{(H,T)}(e,b^n)=n$ for every $n$.  In particular, the map $\Gamma(H,T)\to \Gamma(H,T)$ given by $h\mapsto tht^{-1}$ is not a quasi-isometry, and hence the $G$-action on $H$ by conjugation is not an action by quasi-isometries.  
\end{example}

\section{Quasi-isometries  and Hyperbolically embedded subgroups}
\label{sec:Summary}

In this section, we will prove \Cref{propx:summary}.  The theorem is obtained by putting together a simple characterization of hyperbolically embedded subgroups in terms of coned-off Cayley graphs which appeared in work of Rashid and the second author, see~\cite[Propositions 1.5 and 5.8]{MPR2021}; some results about quasi-isometries of pairs  from~\cite{HuMaSa21a}, and some basic facts about hyperbolically embedded subgroups from~\cite{DGOsin2017}. Below we state these results and then we discuss the proof of \Cref{prop:summary}.

\begin{definition}[Reduced collections]
 A collection of subgroups $\calq$ of a group $H$ is \emph{reduced} if  
for any $P,Q\in \calq$ and $g\in H$, if $P$ and $gQg^{-1}$ are commensurable subgroups then  $P=Q$ and $g\in P$.
\end{definition}

\begin{remark}
An almost malnormal collection is reduced.
\end{remark}

\begin{definition}[Fine]\label{def:fine2}
Let $\Gamma$ be a graph and let $v$ be a vertex of  $\Gamma$. Let \begin{align*}
T_v \Gamma = \{w \in V(\Gamma) \mid \{v,w\}\in E(\Gamma)\}.
\end{align*}
denote the set of the  vertices adjacent to $v$.
For $x,y \in T_v \Gamma$,  the \emph{angle metric} $\angle_v (x,y)$ is the length of the shortest path in the graph  $\Gamma \setminus \{v\}$ between $x$ and $y$, with $\angle_v (x,y) = \infty$ if there is no such path.  The graph $\Gamma$ is \emph{fine at $v$} if $(T_v \Gamma,\angle_v)$ is a locally finite metric space. The graph $\Gamma$ is \emph{fine at    $C \subseteq V(\Gamma)$} if $\Gamma$ is fine at $v$  for all $v \in C$.  
\end{definition} 

\begin{definition}[Coned-off Cayley graph]\label{def:conedOffCayleyGraph}
Let $G$ be a group, let $\calp$ be an arbitrary collection of subgroups of $G$, and let $S$ be a subset of $G$. Denote by $G/\calp$ the set of all cosets $gP$ with $g\in G$ and $P\in\calP$. The \emph{coned-off Cayley graph  of $G$ with respect to $\calp$} is the graph $\hat\Gamma(G,\calp,S)$ with vertex set $G\cup G/\calp$ and edges are of the following type
\begin{itemize}
    \item $\{g,gs\}$ for $s\in S$,
    \item $\{x, gP\}$ for $g\in G$, $P\in \calp$ and $x\in gP$.
\end{itemize}
We call vertices of the form $gP$ \emph{cone points}. 
\end{definition}

\begin{prop}\emph{\cite{MPR2021}} \label{prop:HypEmbHatGamma}
Let $\calp$ be a collection of  infinite subgroups of $G$ and let $S$ be a subset of $G$. Then $\calp \hookrightarrow_h (G,S) $ if and only if the Coned-off Cayley graph $\hat{\Gamma}(G,\calp,S)$ is a connected hyperbolic graph which is fine at every cone vertex. 
\end{prop}

\begin{proposition}\emph{\cite[Proposition~5.6]{HuMaSa21a}}\label{prop:FineQI-2}
Let  $G$ and $H$ be groups, let $S\subset G$ and $T\subset H$, and let $S_0\subset S$ and $T_0\subset T$ be finite generating sets of $G$ and $H$ respectively. Consider collections $\calP$ and $\mathcal Q$ of subgroups of $G$ and $H$ respectively. Let  $q\colon G\to H$ be a function.

Suppose $q$ is a quasi-isometry $\Gamma(G,S) \to \Gamma(H,T)$, is a 
  quasi-isometry of pairs $(G, \mathcal{P},S_0)\to (H, \mathcal{Q},T_0)$, and 
   $\dot{q}$ is a bijection $G/\calp \to H/\calq$. 
   \begin{enumerate}
   \item Let $\hat q = q\cup \dot{q}$, then $\hat q$ is a   quasi-isometry   $\hat \Gamma (G,\calp, S) \to \hat   \Gamma (H,\calq, T)$.\label{prop:FineQI-2.1}
   \item If $\hat \Gamma (H,\calq, T)$ is fine at cone vertices, then $\hat \Gamma (G,\calp, S)$ is fine at cone vertices.\label{prop:FineQI-2.2}
   \item If $\calq\hookrightarrow_h(H,T)$, then $\calp\hookrightarrow_h(G,S)$.\label{prop:FineQI-2.3}
   \end{enumerate}
 \end{proposition}

Items \eqref{prop:FineQI-2.1} and \eqref{prop:FineQI-2.2} of Proposition~\ref{prop:FineQI-2} are taken from~\cite[Proposition~5.6]{HuMaSa21a}, and the last item is a direct consequence of Proposition~\ref{prop:HypEmbHatGamma}.

 \begin{proposition}\emph{\cite[Proposition 5.12]{HuMaSa21a}\label{prop:DotqIsFunction}}
 Let $q\colon (G, \mathcal{P})\to (H, \mathcal{Q})$ be a $(L,C,M)$-quasi-isometry of  pairs. Then 
 \begin{enumerate}
     \item $\dot{q}$ is a surjective function $G/\calp \to H/\calq$ if $\calq$ is reduced. 
     \item $\dot{q}$ is a bijection $G/\calp \to H/\calq$ if $\calp$ and $\calq$ are reduced. 
 \end{enumerate}
 \end{proposition}

\begin{proposition}\emph{\cite[Proposition 6.2]{HuMaSa21a}}\label{lem:QI-Refinement}
Let $\calp^\ast$ be a refinement of a finite collection of subgroups $\calp$ of a finitely generated group $G$. If $P$ is a finite index subgroup of $\Comm_G(P)$ for every $P\in\calp$, then $(G,\calp)$ and $(G,\calp^*)$ are quasi-isometric pairs via the identity map on $G$. 
\end{proposition}

\begin{proposition}\emph{\cite[Proposition 6.7]{HuMaSa21a}}\label{prop:MalnormalityQIinvariance}
 Let $q\colon (G, \mathcal{P})\to (H, \mathcal{Q})$ be a quasi-isometry of pairs.  If $\calq$ is an almost malnormal finite collection of infinite  subgroups and $\calp$ is a finite collection, then any refinement $\calp^*$ of $\calp$ is almost malnormal.
\end{proposition}

\begin{proposition}\emph{\cite[Proposition 4.33]{DGOsin2017}}\label{lem:malnormal}
Let $\calp$ be a collection of subgroups of a group $G$. If $\calp\hookrightarrow_h G$ then $\calp$ is an almost malnormal collection.
 \end{proposition}

 We are now ready to prove \Cref{propx:summary}.
 
\begin{theorem}[\Cref{propx:summary}]\label{prop:summary}
Let $q\colon G\to H$ be a quasi-isometry of finitely generated groups, let $\calp$ and $\calq$ be finite collections of subgroups of $G$ and $H$ respectively, and let $S$ and $T$ be (not necessarily finite) generating sets of $G$ and $H$ respectively. Suppose 
\begin{enumerate}
    \item $q\colon (G,\calp) \to (H,\calq)$ is a quasi-isometry of pairs, and
    \item $q\colon  \Gamma(G,S) \to  \Gamma(H,T)$ is a quasi-isometry.
\end{enumerate}
The following statements hold:
\begin{enumerate}
    \item If $\calp$ and $\calq$ are reduced collections in $G$ and $H$ respectively; then $\calp\hookrightarrow_h (G,S)$ if and only if $\calq\hookrightarrow_h (H,T)$.
    \item If $\calq$ contains only infinite subgroups and  $\calq\hookrightarrow_h (H,T)$ then $\calp^*\hookrightarrow_h (G,S)$.
\end{enumerate}
\end{theorem}
\begin{proof}
For the first statement, since $\calp$ and $\calq$ are reduced,   \Cref{prop:DotqIsFunction} implies that   $\dot q\colon G/\calp \to H/\calq$ is a bijection. Then \Cref{prop:FineQI-2} implies that   $\hat\Gamma(G,\calp,S)$ is hyperbolic and fine at cone vertices if and only if $\hat\Gamma(H,\calq,T)$ is hyperbolic and fine at cone vertices. Then \Cref{prop:HypEmbHatGamma} concludes the proof of the first statement.

The second statement is a consequence of the first statement as follows. 
That $\calq\hookrightarrow_h H$ implies that $\calq$ is an almost malnormal collection of subgroups in $H$, see \Cref{lem:malnormal}. It follows that  $\calq$ is reduced in $H$. Then, since $\calq$ contains only infinite subgroups, \Cref{prop:MalnormalityQIinvariance} implies that $\calp^\ast$ is reduced. By  \Cref{lem:QI-Refinement}, $q\colon (G,\calp^\ast)\to (H,\calq)$ is a quasi-isometry of pairs. Then $\calq\hookrightarrow_h H$ and the first statement of the proposition imply that $\calp^\ast \hookrightarrow (G,S)$. 
\end{proof}


\section{Hyperbolically embedded subgroups and commensurability}
\label{sec:ThmE}

In this section we prove \Cref{corx:Saturday}. The argument    uses the following proposition which is a strengthening of~\cite[Proposition 2.15]{MaSa21}. It essentially follows from the proof in the cited article; but we have included the proof for the convenience of the reader.

\begin{proposition}\label{prop:FiniteIndexSupergroupIff}
Let $H$ be a finite index subgroup of a finitely generated group $G$, and let $\mathcal{Q}$ be a finite collection of subgroups of $H$.
The following statements are equivalent:
\begin{enumerate}
    \item The inclusion $H\hookrightarrow G$ is a quasi-isometry of pairs $(H,\mathcal{Q}) \hookrightarrow (G, \mathcal{Q})$.
    
    \item For any $Q\in \mathcal{Q}$ and   $g\in G$, there is  $Q'\in \mathcal{Q}$ and $h\in H$ such that $\Hdist_G(gQ  , hQ')<\infty$. 
\end{enumerate}
\end{proposition}
\begin{proof}
That (1) implies (2) is trivial.  Assume statement (2). Since $H$ is a finite index subgroup of the finitely generated group $G$, assume  $H\hookrightarrow G$ is an $(L,C)$ quasi-isometry. Since $H$ is finite index in $G$, and $\mathcal{Q}$ is a finite collection,  the $H$-action on $G/\mathcal{Q}$ has finitely many orbits. For $gQ\in G/\mathcal{Q}$, let
\[\Hdist_G(gQ, H/\mathcal{Q}):= \min\left\{\Hdist_G(gQ, hQ') \colon hQ'\in H/\mathcal{Q} \right\}.\]
Let $\mathcal{R}$ be a finite collection of orbit representatives of the $H$-action on $G/\mathcal{Q}$.  By hypothesis, for $gQ \in \mathcal{R}$
 there is $h Q'\in H/ \mathcal{Q}$ such that $\Hdist(gQ,hQ')<\infty$ and therefore 
\[ M = \max\{\Hdist_G(gQ,H/\mathcal{Q}) \colon gQ \in \mathcal{R}\} <\infty \]
is a well defined integer since $\mathcal{R}$ is a finite set. Since the subset $H/\mathcal{Q}$ of $G/\mathcal{Q}$ is $H$-invariant,
\[\Hdist_G(gQ, H/\mathcal{Q}) = \Hdist_G(hgQ, H/\mathcal{Q}) \]
for every $gQ\in \mathcal{R}$ and $h\in H$. Since $\mathcal{R}$ is a collection of representatives of orbits of $G/\mathcal{Q}$, 
\[\Hdist_G(gQ, H/\mathcal{Q}) \leq M\]
for every $gQ \in G/
\mathcal{Q}$. Hence $(H, \mathcal{Q}) \hookrightarrow (G, \mathcal{Q})$ is an $(L,C,M)$  quasi-isometry of pairs .
\end{proof} 

\begin{remark}\label{remark.qi}
Let $G$ be a group and let $T$ and $S$ generating sets with finite symmetric difference. Then the identity map on $G$ is a quasi-isometry $\Gamma(G,T) \to \Gamma(G,S)$. 
\end{remark}

\begin{theorem}[\Cref{corx:Saturday}] \label{cor:Saturday}
Let $H$ be a finite index normal subgroup of a finitely generated group $G$, and let $\calq$ be a finite collection of infinite subgroups of $H$ such that $\calq\hookrightarrow_h (H,T)$.  
Suppose:  
\begin{enumerate}
    \item  The $G$-action by conjugation on $H$ is a uniform quasi-action on  $\Gamma(H,T)$.
    \item The collection  $\{hQh^{-1} \colon h\in H \text{ and } Q\in \calq\}$ is invariant under conjugation by $G$. 
\end{enumerate}
If $\calq^*$ is a refinement of $\calq$ in $G$ and $R$ is a transversal of $H$ in $G$, then $\calq^*\hookrightarrow_h (G, T\cup R)$.
\end{theorem}

\begin{proof}
Since $H$ is finitely generated, by adding a finitely many elements we can assume that $T$ generates $H$.  Note that this preserves  $\calq\hookrightarrow_h (H,T)$ by  \cite[Cor. 4.27]{DGOsin2017}, and the quasi-isometry type of $\Gamma(H,T)$ by \Cref{remark.qi}. Under this assumption, the conclusion will follow from the second statement of \Cref{prop:summary} applied to the  quasi-isometry of finitely generated groups given by the inclusion $H\hookrightarrow G$.

Since $\calq\hookrightarrow_h (H,T)$, $\calq$ is an almost malnormal collection, see \Cref{lem:malnormal}. The assumption that $\calq$ consist only of infinite subgroups implies that for any $Q\in\calq$,
\[Q=\Comm_H(Q)=\Comm_G(Q)\cap H.\]
Since $H$ is finite index in $G$, we have that $Q$ is finite index in $\Comm_G(Q)$. Then,   \Cref{lem:QI-Refinement} implies  that the identity map on $G$ is a quasi-isometry of pairs $(G,\calq) \xrightarrow{}(G,\calq^\ast)$. On the other hand, since the collection $\{hQh^{-1}\colon h\in H \text{ and } Q\in\calq\}$ is invariant under conjugation by elements of $G$, we have for any $g\in G$ and $Q\in\calq$ there is $h\in H$ such that $gQg^{-1}=hQ'h^{-1}$ and hence
\[\Hdist_G(gQ,hQ')\leq \Hdist_G(gQ,Q^g)+ \Hdist_G(Q^g,(Q')^h)+ \Hdist((Q')^h,hQ')<\infty.\]
\Cref{prop:FiniteIndexSupergroupIff}  implies that   $H\hookrightarrow G$ is a quasi-isometry of pairs  $(H,\calq)\to (G,\calq)$. It follows that  $H\hookrightarrow G$ is a quasi-isometry of pairs $(H,\calq) \to (G,\calq^\ast)$ as it is the composition $(H,\calq)\hookrightarrow (G,\calq) \xrightarrow{}(G,\calq^\ast)$. Let $R$ be a  transversal of $H$ in $G$ and let $S=T\cup R$. Since the $G$-action by conjugation on $H$ is uniform on $\Gamma(H,T)$, \Cref{prop:FridayNight} implies that   $H\hookrightarrow G$ is a quasi-isometry  $\Gamma(H,T) \to (G,S)$. The hypothesis of \Cref{prop:summary} has been verified and therefore,  $\calq\hookrightarrow_h (H,T)$     implies  $\calq^*\hookrightarrow_h (G,S)$.
\end{proof}

\section{Semi-direct products and Hyperbolically embedded subgroups}
\label{sec:ThmF}

In this section we will prove \Cref{thmx.hypemb.semidirect} about semi-direct products.  The hypothesis of the following proposition and theorem reflects the issues posed by the example of Minasyan and Osin (\Cref{ex.MinasyanOsin}).

\begin{prop}\label{prop.semidirect.qi02}
Let $A$ be a group with (not necessarily finite) generating set $T$, let  $\mathcal{H}$ be a  collection of subgroups, and let $F\leq \mathsf{Aut}(A)$ be a finite subgroup. Suppose that $T$ and $\mathcal{H}$ are $F$-invariant, and the $F$-action on $\mathcal{H}$ is free. Let $\calh_F$ be a collection of representatives of $F$-orbits in $\calh$.  Then the inclusion $A\hookrightarrow A\rtimes F$ induces:
\begin{enumerate}
    \item a quasi-isometry $\Gamma(A,T) \to \Gamma(A\rtimes F, T\cup F)$;
    
   \item  and, if $A$ is finitely generated, a quasi-isometry of pairs $(A,\mathcal{H}) \to (A\rtimes F,\mathcal{H}_F)$.
    
\end{enumerate}
\end{prop}
\begin{proof}
To prove the first statement, let $S=T\cup F$ and let $\dist_{T}$ and $\dist_S$ be the word metrics on $A$ and $A\rtimes F$  induced by $T$, and $S$ respectively. Let $q
\colon A \hookrightarrow A\rtimes F$ be the inclusion, and let  $\bar q\colon   A\rtimes F \to A$ such that for $a\in A$ and $f\in F$,   $\bar q(af)=a$. Note that $\bar q$ is a well defined $A$-equivariant map since each element of $A\rtimes F$ can be expressed as a product $af$ in a unique way. Observe that $\bar q\circ q$ is the identity on $A$, and $q\circ \bar q$ is at distance one from the identity map on $A\rtimes F$ with respect to $\dist_S$.  Since the Cayley graph $\Gamma(A,T)$ is a subgraph of $\Gamma(A\rtimes F, T\cup F)$, it is immediate that for any $u,v\in A$, $\dist_S(q(u),q(v))\leq \dist_T(u,v)$.
To conclude the proof of the statement, we show that for any $u,v\in A\rtimes F$, $\dist_T(\bar q(u),\bar q(v))\leq \dist_S(u,v)$. Note that it is enough to consider the case that $\dist_S(u,v)=1$.  Let $w_1,w_2\in A\rtimes F$ such that $\dist_S(w_1,w_2)=1$. Then $w_1=a_1f_1$ and $w_2=a_2f_2$ and $\bar q(w_i)=a_i$.  It follows that $g=(a_1f_1)^{-1}a_2f_2\in T\cup F$. Observe that 
\[g= f_1^{-1}a_1^{-1}a_2f_2=(a_1^{-1})^{f_1^{-1}}f_1^{-1}a_2f_2=(a_1^{-1})^{f_1^{-1}}a_2^{f_1^{-1}}f_1^{-1}f_2=(a_1^{-1}a_2)^{f_1^{-1}}f_1^{-1}f_2\in T\cup F.\]
There are two cases, either $g\in T$ or $g\in F$, since $T\cap F=\emptyset$. We regard $T\cup F$ and $F$ as a subset and a subgroup of $A\rtimes F$ respectively. If $g\in T$, then $f_1=f_2$ and hence $(a_1^{-1}a_2)^{f_1^{-1}} \in T$; since $T$ is $F$-invariant, $a_1$  and $a_2$ are adjacent in $\Gamma$, and hence $\dist_T(\bar q(w_1), \bar q(w_2))=1$. If $g\in F$,  then $a_1=a_2$ and hence $\dist_T(\bar q(w_1), \bar q(w_2))=0$.  

For the second statement, suppose that $A$ is finitely  generated and let  $\dist$ denote word metric  on  $A\rtimes F$ induced by  finite generating set, and  let  $\Hdist_{A\rtimes F}$ be the induced Hausdorff distance. 
Let $M=\max_{f\in F}\dist(1, f)$.
Since the inclusion $A\hookrightarrow A\rtimes F$ is a quasi-isometry of finitely generated groups and $\calh_F \subset \calh$, it is enough to   prove that for any   $H\in\calh$ there is a left coset in   $(A\rtimes F)/\calh_F$ at Haudorff distance at most $M$ in $A\rtimes F$.
Let $H\in \calh$. Since the $F$-action on $\calh$ by conjugation is free, there is a unique $f\in F$ and a unique $K\in\calh_F$ such that $H=fKf^{-1}$. Observe that
\[  \Hdist(H, fK) = \Hdist(fKf^{-1}, fK) \leq \dist(1, f^{-1}) \leq M,\]
and this completes the proof.
\end{proof}

\begin{thm}[\Cref{thmx.hypemb.semidirect}]\label{thm.hypemb.semidirect}
Let $A$ be a finitely generated group with (not necessarily finite) generating set $T$, and let $\calh$ be a finite collection of infinite subgroups such that  $\mathcal{H}\hookrightarrow_h (A,T)$. If $F\leq \mathsf{Aut}(A)$ is finite, $T$ and $\calh$ are   $F$-invariant and the $F$-action on  $\mathcal{H}$ is free, then $\calh_F\hookrightarrow_h(A\rtimes F,T\cup F)$ where $\calh_F$ is collection of representatives of $F$-orbits in $\calh$.
\end{thm}
\begin{proof}
By \Cref{prop.semidirect.qi02}, the inclusion $A\hookrightarrow A\rtimes F$ induces a quasi-isometry $\Gamma(A,T) \to \Gamma(A\rtimes F, T\cup F)$, and a quasi-isometry of pairs $(A,\calh) \to (A\rtimes F, \calh_F)$. Since $\calh\hookrightarrow_h A$, the collection $\calh$ is almost malnormal in $A$; then the assumption that $F$ acts freely on $\mathcal{H}$ implies that a refinement of $\calh$ in $A\rtimes F$ is $\mathcal{H}_F$, this was observed in \Cref{ex:refinement}.
Since $\calh$ contains only infinite subgroups and $\calh\hookrightarrow_h A$,  \Cref{prop:summary} implies that $\calh_F \hookrightarrow_h (A\rtimes F, T\cup F)$.   
\end{proof}

\section{Concluding remarks}\label{sec conclusion}

A positive answer to the following question would allows us to drop the first hypothesis of \Cref{propx:summary} for the relevant groups.

\begin{question}\label{quest qichar}
    Let $G$ be a finitely generated NRH acylindrically hyperbolic group.  Does $G$ contain a q.i.-characteristic collection of hyperbolically embedded subgroups?
\end{question}

It is possible that $\cala\calh$-accessibility as defined in \cite{ABO2019} may be necessary for a positive answer to \Cref{quest qichar}.  Note that this property does not always hold (see \cite{Abbott2016}).
    
It is tempting to weaken the definition of a quasi-isometry of pairs $q\colon (G,\calp)\to (H,\calq)$ to remove the uniform constant $M$ bounding the Hausdorff distances on the cosets and instead ask the relation
\[ \dot{q}=\{(A,B)\in G/\calp\times H/\calq : \Hdist_H(q(A),B)<\infty\}\]
satisfies that the projections into $G/\calp$ and $H/\calq$ are surjective.  We shall call the map $q$ in this modified definition an \emph{almost quasi-isometry of pairs} following \cite[Section~5]{HuMaSa21survey}.

Indeed, there is work of Margolis \cite{Margolis2021} where the main theorems do not require require this additional hypothesis.  However, Margolis shows that the hypotheses assumed in the main results of loc. cit. in fact imply that such a constant $M$ exists (see \cite[Theorem~4.1]{Margolis2021}).  Note that our results in this article rely on the existence of a constant $M$ --- primarily due to the use of \cite[Proposition~5.6]{HuMaSa21a}.  Thus, we raise the following question.

\begin{question}
    Let $G$ and $H$ be finitely generated groups with finite collections of subgroups $\calp$ and $\calq$ respectively.  When is an $(L,C)$-almost quasi-isometry of pairs $q\colon(G,\calp)\to(H,\calq)$ an $(L,C,M)$-quasi-isometry of pairs?
\end{question}

Motivated by results of \cite{BHS21}, the referee of the article suggested that it might be interesting to investigate other  relaxations of the definition of a quasi-isometry of pairs (\Cref{defn:quasi-isometry-pairs}). For example in the sense that the image of every element of the collection $\mathcal{A}$ lies at uniform Hausdorff
distance of the union of finitely many elements in the collection $\mathcal{B}$. Having a more general notion could allow a broader strategy towards tackling \Cref{question.1} based on the methods in this article.



\AtNextBibliography{\small}
\printbibliography

\end{document}